\documentclass{amsart}
\usepackage{amsfonts}

\newtheorem{theorem}{Theorem}
\theoremstyle{plain}

\newtheorem{corollary}{Corollary}

\newtheorem{example}{Example}

\newtheorem{lemma}{Lemma}

\numberwithin{equation}{section}
\input{tcilatex}

\begin{document}
\title[Ideals in intra-regular left almost semigroups]{Ideals in
intra-regular left almost semigroups}
\subjclass[2000]{20M10 and 20N99}
\author[M. Khan, V. Amjid and Faisal ]{}
\maketitle

\begin{center}
$^{\ast }$\textbf{Madad Khan, Venus Amjid and Faisal }

\bigskip

\textbf{Department of Mathematics}

\textbf{COMSATS Institute of Information Technology}

\textbf{Abbottabad, Pakistan.}\bigskip

$^{\ast }$\textbf{E-mail: madadmath@yahoo.com, E-mail: venusmath@yahoo.com}

\textbf{E-mail: yousafzaimath@yahoo.com}

\bigskip
\end{center}

\textbf{Abstract.} In this paper, we have introduced the notion of $(1,2)$%
-ideal in an LA-semigroup and shown that $(1,2)$-ideal and two-sided ideal
coincide in an intra-regular LA-semigroup. We have characterized an
intra-regular LA-semigroup by using the properties of left and right ideals.
Some natural examples of LA-semigroups have been given. Further we have
investigated some useful conditions for an LA-semigroup to become an
intra-regular LA-semigroup and given the counter examples to illustrate the
converse inclusions. All the ideals (left, right, two-sided, interior,
quasi, bi- generalized bi- and $(1,2))$ of an intra-regular LA-semigroup
have been characterized. Finally we have given an equivalent statement for a
two-sided ideal of an intra-regular LA-semigroup in terms of the
intersection of two minimal two-sided ideals of an intra-regular
LA-semigroup.

\textbf{Keywords}. LA-semigroups, intra-regular LA-semigroups and $(1,2)$%
-ideals.

\begin{center}
\bigskip

{\LARGE Introduction}
\end{center}

The idea of generalization of a commutative semigroup was first introduced
by Kazim and Naseeruddin in $1972$ $($see \cite{kaz}$)$. They named it as a
left almost semigroup $($LA-semigroup$)$. It is also called an
Abel-Grassmann's groupoid (AG-groupoid) \cite{ref10}.

An LA-semigroup is a non-associative and non-commutative algebraic structure
mid way between a groupoid and a commutative semigroup. This structure is
closely related with a commutative semigroup, because if an LA-semigroup
contains a right identity, then it becomes a commutative semigroup \cite%
{Mus3}. The connection of a commutative inverse semigroup with an
LA-semigroup has been given in \cite{Mushtaq and Yusuf} as, a commutative
inverse semigroup $(S$, $\circ )$ becomes an LA-semigroup $(S$, $\cdot )$
under $a\cdot b=b\circ a^{-1},$ for all $a,b\in S$ . An LA-semigroup $S$
with left identity becomes a semigroup $(S$, $\circ )$ defined as, for all $%
x $, $y\in S$, there exists $a\in S$ such that $x\circ y=(xa)y$ \cite{Protic}%
. An LA-semigroup is the generalization of a semigroup theory \cite{Mus3}
and has vast applications in collaboration with semigroup like other
branches of mathematics.

An LA-semigroup is a groupoid $S$ whose elements satisfy the left invertive
law $(ab)c=(cb)a$, for\ all $a,b,c\in S.$ In an LA-semigroup, the medial law 
\cite{kaz} $(ab)(cd)=(ac)(bd)$ holds for\ all $a,b,c,d\in S$. An
LA-semigroup may or may not contains a left identity. The left identity of
an LA-semigroup allow us to introduce the inverses of elements in an
LA-semigroup. If an LA-semigroup contains a left identity, then it is unique 
\cite{Mus3}. In an LA-semigroup $S$ with left identity, the paramedial law $%
(ab)(cd)=(dc)(ba)$ holds for\ all $a,b,c,d\in S.$ If an LA-semigroup
contains a left identity, then by using medial law, we get $a(bc)=b(ac)$,
for\ all $a,b,c\in S.$ Several examples and interesting properties of
LA-semigroups can be found in \cite{Mus3} and \cite{Protic}.

In this paper, we have extended the concept of an intra-regular LA-semigroup
first considered by M. Khan and N. Ahmad in \cite{nav}.

\bigskip

Let $S$ be an LA-semigroup. By an LA-subsemigroup of $S,$ we means a
non-empty subset $A$ of $S$ such that $A^{2}\subseteq A$.

A non-empty subset $A$ of an LA-semigroup $S$ is called a left (right) ideal
of $S$ if $SA\subseteq A$ $(AS\subseteq A).$

By two-sided ideal or simply ideal, we mean a non-empty subset of an
LA-semigroup $S$ which is both a left and a right ideal of $S$.

A non empty subset $A$ of an LA-semigroup $S$ is called a generalized
bi-ideal of $S$ if $(AS)A\subseteq A$ and an LA-subsemigroup $A$ of $S$ is
called a bi-ideal of $S$ if $(AS)A\subseteq A$.

A non-empty subset $A$ of an LA-semigroup $S$ is called an interior ideal of 
$S$ if $(SA)S\subseteq A$.

A non empty subset $A$ of an LA-semigroup $S$ is called a quasi ideal of $S$
if $SA\cap AS\subseteq A.$

An LA-subsemigroup $A$ of $S$ is called a $(1,2)$-ideal of $S$ if $%
(AS)A^{2}\subseteq A$.

\begin{example}
\label{exp}Let us consider an LA-semigroup $S=\left \{ a,b,c,d,e,f\right \} $
with left identity $e$ in the following Clayey's table.
\end{example}

\begin{center}
\begin{tabular}{c|cccccc}
. & $a$ & $b$ & $c$ & $d$ & $e$ & $f$ \\ \hline
$a$ & $a$ & $a$ & $a$ & $a$ & $a$ & $a$ \\ 
$b$ & $a$ & $b$ & $b$ & $b$ & $b$ & $b$ \\ 
$c$ & $a$ & $b$ & $f$ & $f$ & $d$ & $f$ \\ 
$d$ & $a$ & $b$ & $f$ & $f$ & $c$ & $f$ \\ 
$e$ & $a$ & $b$ & $c$ & $d$ & $e$ & $f$ \\ 
$f$ & $a$ & $b$ & $f$ & $f$ & $f$ & $f$%
\end{tabular}
\end{center}

\begin{example}
Let us consider the set $(%
\mathbb{R}
,+)$ of all real numbers under the binary operation of addition. If we
define $a\ast b=b-a-r,$ where $a,b,r\in R,$ then $(%
\mathbb{R}
,\ast )$ becomes an LA-semigroup as,%
\begin{equation*}
(a\ast b)\ast c=c-(a\ast b)-r=c-(b-a-r)-r=c-b+a+r-r=c-b+a
\end{equation*}%
and%
\begin{equation*}
(c\ast b)\ast a=a-(c\ast b)-r=a-(b-c-r)-r=a-b+c+r-r=a-b+c.
\end{equation*}%
Since $(%
\mathbb{R}
,+)$ is commutative so $(a\ast b)\ast c=(c\ast b)\ast a$ and therefore $(%
\mathbb{R}
,\ast )$ satisfies a left invertive law. It is easy to observe that $(R,\ast
)$ is non-commutative and non-associative. The same is hold for set of
integers and rationals. Thus $(%
\mathbb{R}
,\ast )$ is an LA-semigroup which is the generalization of an LA-semigroup
given in $1988$ $($see \cite{Mushtaq and Yusuf}$)$. Similarly if we define $%
a\ast b=ba^{-1}r^{-1},$ then $(%
\mathbb{R}
\backslash \{0\},\ast )$ becomes an LA-semigroup and the same holds for the
set of integers and rationals. This LA-semigroup is also the generalization
of an LA-semigroup given in $1988$ $($see \cite{Mushtaq and Yusuf}$)$.
\end{example}

An element $a$ of an LA-semigroup $S$ is called an intra-regular if there
exist $x,y\in S$ such that $a=(xa^{2})y$ and $S$ is called intra-regular, if
every element of $S$ is intra-regular.

\begin{example}
\label{tb}Let $S=\{a,b,c,d,e\}$ be an LA-semigroup with left identity $b$ in
the following multiplication table.
\end{example}

\begin{center}
\begin{tabular}{l|lllll}
. & $a$ & $b$ & $c$ & $d$ & $e$ \\ \hline
$a$ & $a$ & $a$ & $a$ & $a$ & $a$ \\ 
$b$ & $a$ & $b$ & $c$ & $d$ & $e$ \\ 
$c$ & $a$ & $e$ & $b$ & $c$ & $d$ \\ 
$d$ & $a$ & $d$ & $e$ & $b$ & $c$ \\ 
$e$ & $a$ & $c$ & $d$ & $e$ & $b$%
\end{tabular}
\end{center}

Clearly $S$ is intra-regular because, $a=(aa^{2})a,$ $b=(cb^{2})e,$ $%
c=(dc^{2})e,$ $d=(cd^{2})c,$ $e=(be^{2})e.$

An element $a$ of an LA-semigroup $S$ with left identity $e$ is called a
left (right) invertible if there exits $x\in S$ such that $xa=e$ ($ax=e$)
and $a$ is called invertible if it is both a left and a right invertible. An
LA-semigroup $S$ is called a left (right) invertible if every element of $S$
is a left (right) invertible and $S$ is called invertible if it is both a
left and a right invertible.

Note that in an LA-semigroup $S$ with left identity, $S=S^{2}.$

\begin{theorem}
\label{1}Every LA-semigroup $S$ with left identity is an intra-regular if $S$
is left (right) invertible.
\end{theorem}

\begin{proof}
Let $S$ be a left invertible LA-semigroup with left identity, then for $a\in
S$ there exists $a^{^{\prime }}\in S$ such that $a^{^{\prime }}a=e.$ Now by
using left invertive law, medial law with left identity and medial law, we
have%
\begin{eqnarray*}
a &=&ea=e(ea)=(a^{^{\prime }}a)(ea)\in (Sa)(Sa)=(Sa)((SS)a) \\
&=&(Sa)((aS)S)=(aS)((Sa)S)=(a(Sa))(SS) \\
&=&(a(Sa))S=(S(aa))S=(Sa^{2})S.
\end{eqnarray*}

Which shows that $S$ is intra-regular. Similarly in the case of right
invertible.
\end{proof}

\begin{theorem}
\label{ji}An LA-semigroup $S$ is intra-regular if $Sa=S$ or $aS=S$ holds for
all $a$ $\in S$.
\end{theorem}

\begin{proof}
Let $S$ be an LA-semigroup such that $Sa=S$ holds for all $a\in S,$ then $%
S=S^{2}$. Let $a\in S$, therefore by using medial law, we have%
\begin{equation*}
a\in S=(SS)S=((Sa)(Sa))S=((SS)(aa))S\subseteq (Sa^{2})S.
\end{equation*}

Which shows that $S$ is intra-regular.

Let $a\in S$ and assume that $aS=S$ holds for all $a\in S,$ then by using
left invertive law, we have%
\begin{equation*}
a\in S=SS=(aS)S=(SS)a=Sa.
\end{equation*}%
Thus $Sa=S$ holds for all $a$ $\in S$, therefore it follows from above that $%
S$ is intra-regular.
\end{proof}

The converse is not true in general from Example \ref{tb}.

\begin{corollary}
If $S$ is an LA-semigroup such that $aS=S$ holds for all $a$ $\in S,$ then $%
Sa=S$ holds for all $a$ $\in S.$
\end{corollary}

\begin{theorem}
\label{ki}If $S$ is intra-regular LA-semigroup with left identity, then $%
(BS)B=B\cap S,$ where $B$ is a bi-$($generalized bi-$)$ ideal of $S$.
\end{theorem}

\begin{proof}
Let $S$ be an intra-regular LA-semigroup with left identity, then clearly $%
(BS)B\subseteq B\cap S$. Now let $b\in $ $B\cap S,$ which implies that $b\in
B$ and $b\in S.$ Since $S$ is intra-regular so there exist $x,y\in S$ such
that $b=(xb^{2})y.$ Now by using medial law with left identity$,$ left
invertive law$,$ paramedial law and medial law$,$ we have%
\begin{eqnarray*}
b &=&(x(bb))y=(b(xb))y=(y(xb))b=(y(x((xb^{2})y)))b \\
&=&(y((xb^{2})(xy)))b=((xb^{2})(y(xy)))b=(((xy)y)(b^{2}x))b \\
&=&((bb)(((xy)y)x))b=((bb)((xy)(xy)))b=((bb)(x^{2}y^{2}))b \\
&=&((y^{2}x^{2})(bb))b=(b((y^{2}x^{2})b))b\in (BS)B.
\end{eqnarray*}

This shows that $(BS)B=B\cap S.$
\end{proof}

The converse is not true in general. For this, let us consider an
LA-semigroup $S$ with left identity $e$ in Example \ref{exp}$.$ It is easy
to see that $\left \{ a,b,f\right \} $ is a bi-$($generalized bi-$)$ ideal
of $S$ such that $(BS)B=B\cap S$ but $S$ is not an intra-regular because $%
d\in S$ is not an intra-regular.

\begin{corollary}
If $S$ is intra-regular LA-semigroup with left identity, then $(BS)B=B,$
where $B$ is a bi-$($generalized bi-$)$ ideal of $S$.
\end{corollary}

\begin{theorem}
\label{aw}If $S$ is intra-regular LA-semigroup with left identity, then $%
(SB)S=S\cap B,$ where $B$ is an interior ideal of $S$.
\end{theorem}

\begin{proof}
Let $S$ be an intra-regular LA-semigroup with left identity, then clearly $%
(SB)S\subseteq S\cap B$. Now let $b\in S\cap B,$ which implies that $b\in S$
and $b\in B.$ Since $S$ is an intra-regular so there exist $x,y\in S$ such
that $b=(xb^{2})y.$ Now by using paramedial law and left invertive law$,$ we
have%
\begin{equation*}
b=((ex)(bb))y=((bb)(xe))y=(((xe)b)b)y\in (SB)S.
\end{equation*}

Which shows that $(SB)S=S\cap B.$
\end{proof}

The converse is not true in general. It is easy to see that form Example \ref%
{exp} that $\left \{ a,b,f\right \} $ is an interior ideal of an
LA-semigroup $S$ with left identity $e$ such that $(SB)S=B\cap S$ but $S$ is
not an intra-regular because $d\in S$ is not an intra-regular.

\begin{corollary}
If $S$ is intra-regular LA-semigroup with left identity, then $(SB)S=B,$
where $B$ is an interior ideal of $S$.
\end{corollary}

Let $S$ be an LA-semigroup, then $\emptyset \neq A\subseteq S$ is called
semiprime if $a^{2}\in A\ $implies $a\in A.$

\begin{theorem}
\label{intr}An LA-semigroup $S$ with left identity is intra-regular if $%
L\cup R=LR,$ where $L$ and $R$ are the left and right ideals of $S$
respectively such that $R$ is semiprime.
\end{theorem}

\begin{proof}
Let $S$ be an LA-semigroup with left identity, then clearly $Sa$ and $a^{2}S$
are the left and right ideals of $S$ such that $a\in Sa$ and $a^{2}\in
a^{2}S,$ because by using paramedial law, we have%
\begin{equation*}
a^{2}S=(aa)(SS)=(SS)(aa)=Sa^{2}.
\end{equation*}

Therefore by given assumption, $a\in a^{2}S$. Now by using left invertive
law, medial law, paramedial law and medial law with left identity, we have%
\begin{eqnarray*}
a &\in &Sa\cup a^{2}S=(Sa)(a^{2}S)=(Sa)((aa)S)=(Sa)((Sa)(ea)) \\
&\subseteq &(Sa)((Sa)(Sa))=(Sa)((SS)(aa))\subseteq (Sa)((SS)(Sa)) \\
&=&(Sa)((aS)(SS))=(Sa)((aS)S)=(aS)((Sa)S) \\
&=&(a(Sa))(SS)=(a(Sa))S=(S(aa))S=(Sa^{2})S.
\end{eqnarray*}

Which shows that $S$ is intra-regular.
\end{proof}

The converse is not true in general. In Example \ref{exp}, the only left and
right ideal of $S$ is $\{a,b\}$, where $\{a,b\}$ is semiprime such that $%
\{a,b\} \cup \{a,b\}=\{a,b\} \{a,b\}$ but $S$ is not an intra-regular
because $d\in S$ is not an intra-regular.

\begin{lemma}
\cite{nav} \label{jk}If $S$ is intra-regular regular LA-semigroup, then $%
S=S^{2}$.
\end{lemma}

\begin{theorem}
For a left invertible LA-semigroup $S$ with left identity, the following
conditions are equivalent.
\end{theorem}

$(i)$ $S$ is intra-regular.

$(ii)$ $R\cap L=RL,$ where $R$ and $L$ are any left and right ideals of $S$
respectively.

\begin{proof}
$(i)\Longrightarrow (ii):$ Assume that $S\ $is intra-regular LA-semigroup$\ $%
with left identity and let $a\in S,$ then there exist $x,y\in S$ such that $%
a=(xa^{2})y.$ Let $R$ and $L$ be any left and right ideals of $S$
respectively, then obviously $RL\subseteq R\cap L.$ Now let $a\in R\cap L$
implies that $a\in R$ and $a\in L.$ Now by using medial law with left
identity, medial law and left invertive law, we have%
\begin{eqnarray*}
a &=&(xa^{2})y\in (Sa^{2})S=(S(aa))S=(a(Sa))S=(a(Sa))(SS) \\
&=&(aS)((Sa)S)=(Sa)((aS)S)=(Sa)((SS)a)=(Sa)(Sa) \\
&\subseteq &(SR)(SL)=((SS)R)(SL)=((RS)S)(SL)\subseteq RL.
\end{eqnarray*}

This shows that $R\cap L=RL.$

$(ii)\Longrightarrow (i):$ Let $S$ be a left invertible LA-semigroup with
left identity, then for $a\in S$ there exists $a^{^{\prime }}\in S$ such
that $a^{^{\prime }}a=e.$ Since $a^{2}S$ is a right ideal and also a left
ideal of $S\ $such that $a^{2}\in a^{2}S$, therefore by using given
assumption, medial law with left identity and left invertive law, we have%
\begin{eqnarray*}
a^{2} &\in &a^{2}S\cap
a^{2}S=(a^{2}S)(a^{2}S)=a^{2}((a^{2}S)S)=a^{2}((SS)a^{2}) \\
&=&(aa)(Sa^{2})=((Sa^{2})a)a.
\end{eqnarray*}

Thus we get, $a^{2}=((xa^{2})a)a$ for some $x\in S.$

Now by using left invertive law, we have%
\begin{eqnarray*}
(aa)a^{^{\prime }} &=&(((xa^{2})a)a)a^{^{\prime }} \\
(a^{^{\prime }}a)a &=&(a^{^{\prime }}a)(((xa^{2})a) \\
a &=&(xa^{2})a.
\end{eqnarray*}

This shows that $S$ is intra-regular.
\end{proof}

\begin{lemma}
\cite{nav} \label{ideal idemp}Every two-sided ideal of an intra-regular
LA-semigroup $S$ with left identity is idempotent.
\end{lemma}

\begin{theorem}
In an LA-semigroup $S$ with left identity, the following conditions are
equivalent.
\end{theorem}

$(i)$ $S$ is intra-regular.

$(ii)$ $A=(SA)^{2},$ where $A$ is any left ideal of S.

\begin{proof}
$(i)$ $\Longrightarrow (ii):$ Let $A$ be a left ideal of an intra-regular
LA-semigroup $S$ with left identity$,$ then $SA\subseteq A$ and by Lemma \ref%
{ideal idemp}, $(SA)^{2}=SA\subseteq A.$ Now $A=AA\subseteq SA=(SA)^{2},$
which implies that $A=(SA)^{2}.$

$(ii)$ $\Longrightarrow (i):$ Let $A$ be a left ideal of $S,$ then $%
A=(SA)^{2}\subseteq A^{2},$ which implies that $A$ is idempotent and by
using Lemma \ref{iffff}, $S$ is intra-regular.
\end{proof}

\begin{theorem}
In an intra-regular LA-semigroup $S$ with left identity, the following
conditions are equivalent.
\end{theorem}

$(i)$ $A$ is a bi-(generalized bi-) ideal of $S$.

$(ii)$ $(AS)A=A$ and $A^{2}=A.$

\begin{proof}
$(i)\Longrightarrow (ii):$ Let $A$ be a bi-ideal of an intra-regular
LA-semigroup $S$ with left identity$,$ then $(AS)A\subseteq A$. Let $a\in A$%
, then since $S$ is intra-regular so there exist $x,$ $y\in S$ such that $%
a=(xa^{2})y.$ Now by using medial law with left identity$,$ left invertive
law$,$ medial law and paramedial law$,$ we have%
\begin{eqnarray*}
a &=&(xa^{2})y=(x(aa))y=(a(xa))y=(y(xa))a \\
&=&(y(x((xa^{2})y)))a=(y((xa^{2})(xy)))a \\
&=&((xa^{2})(y(xy)))a=((x(aa))(y(xy)))a \\
&=&((a(xa))(y(xy)))a=((ay)((xa)(xy)))a \\
&=&((xa)((ay)(xy)))a=((xa)((ax)y^{2}))a \\
&=&((y^{2}(ax))(ax))a=(a((y^{2}(ax))x))a\in (AS)A.
\end{eqnarray*}

Thus $(AS)A=A$ holds. Now by using medial law with left identity$,$ left
invertive law$,$ paramedial law and medial law$,$ we have%
\begin{eqnarray*}
a &=&(xa^{2})y=(x(aa))y=(a(xa))y=(y(xa))a=(y(x((xa^{2})y)))a \\
&=&(y((xa^{2})(xy)))a=((xa^{2})(y(xy)))a=((x(aa))(y(xy)))a \\
&=&((a(xa))(y(xy)))a=(((y(xy))(xa))a)a=(((ax)((xy)y))a)a \\
&=&(((ax)(y^{2}x))a)a=(((ay^{2})(xx))a)a=(((ay^{2})x^{2})a)a \\
&=&(((x^{2}y^{2})a)a)a=(((x^{2}y^{2})((x(aa))y))a)a \\
&=&(((x^{2}y^{2})((a(xa))y))a)a=(((x^{2}(a(xa)))(y^{2}y))a)a \\
&=&(((a(x^{2}(xa)))y^{3})a)a=(((a((xx)(xa)))y^{3})a)a \\
&=&(((a((ax)(xx)))y^{3})a)a=((((ax)(ax^{2}))y^{3})a)a \\
&=&((((aa)(xx^{2}))y^{3})a)a=(((y^{3}x^{3})(aa))a)a \\
&=&((a((y^{3}x^{3})a))a)a\subseteq ((AS)A)A\subseteq AA=A^{2}.
\end{eqnarray*}

Hence $A=A^{2}$ holds.

$(ii)\Longrightarrow (i)$ is obvious.
\end{proof}

\begin{theorem}
In an intra-regular LA-semigroup $S$ with left identity, the following
conditions are equivalent.
\end{theorem}

$(i)$ $A$ is a quasi ideal of $S$.

$(ii)$ $SQ\cap QS=Q.$

\begin{proof}
$(i)\Longrightarrow (ii):$ Let $Q$ be a quasi ideal of an intra-regular
LA-semigroup $S$ with left identity$,$ then $SQ\cap QS\subseteq Q$. Let $%
q\in Q$, then since $S$ is intra-regular so there exist $x$, $y\in S$ such
that $q=(xq^{2})y.$ Let $pq\in SQ,$ then by using medial law with left
identity$,$ medial law and paramedial law$,$ we have%
\begin{eqnarray*}
pq &=&p((xq^{2})y)=(xq^{2})(py)=(x(qq))(py)=(q(xq))(py) \\
&=&(qp)((xq)y)=(xq)((qp)y)=(y(qp))(qx) \\
&=&q((y(qp))x)\in QS.
\end{eqnarray*}

Now let $qy\in QS,$ then by using left invertive law, medial law with left
identity and paramedial law$,$ we have%
\begin{eqnarray*}
qp &=&((xq^{2})y)p=(py)(xq^{2})=(py)(x(qq))=x((py)(qq)) \\
&=&x((qq)(yp))=(qq)(x(yp))=((x(yp))q)q\in SQ.
\end{eqnarray*}

Hence $QS=SQ.$ As by using medial law with left identity and left invertive
law$,$ we have%
\begin{equation*}
q=(xq^{2})y=(x(qq))y=(q(xq))y=(y(xq))q\in SQ.
\end{equation*}%
Thus $q\in SQ\cap QS$ implies that $SQ\cap QS=Q$.

$(ii)\Longrightarrow (i)$ is obvious.
\end{proof}

\begin{theorem}
In an intra-regular LA-semigroup $S$ with left identity, the following
conditions are equivalent.
\end{theorem}

$(i)$ $A$ is an interior ideal of $S$.

$(ii)$ $(SA)S=A.$

\begin{proof}
$(i)$ $\Longrightarrow (ii):$ Let $A$ be an interior ideal of an
intra-regular LA-semigroup $S$ with left identity$,$ then $(SA)S\subseteq A$%
. Let $a\in A$, then since $S$ is intra-regular so there exist $x,$ $y\in S$
such that $a=(xa^{2})y.$ Now by using medial law with left identity, left
invertive law and paramedial law$,$ we have%
\begin{eqnarray*}
a &=&(xa^{2})y=(x(aa))y=(a(xa))y=(y(xa))a=(y(xa))((xa^{2})y) \\
&=&(((xa^{2})y)(xa))y=((ax)(y(xa^{2})))y=(((y(xa^{2}))x)a)y\in (SA)S.
\end{eqnarray*}

Thus $(SA)S=A.$

$(ii)$ $\Longrightarrow (i)$ is obvious.
\end{proof}

\begin{theorem}
In an intra-regular LA-semigroup $S$ with left identity, the following
conditions are equivalent.
\end{theorem}

$(i)$ $A$ is a $(1,2)$-ideal of $S$.

$(ii)$ $(AS)A^{2}=A$ and $A^{2}=A$ .

\begin{proof}
$(i)$ $\Longrightarrow (ii):$ Let $A$ be a $(1,2)$-ideal of an intra-regular
LA-semigroup $S$ with left identity$,$ then $(AS)A^{2}\subseteq A$ and $%
A^{2}\subseteq A$. Let $a\in A$, then since $S$ is intra-regular so there
exist $x,$ $y\in S$ such that $a=(xa^{2})y.$ Now by using medial law with
left identity, left invertive law and paramedial law$,$ we have%
\begin{eqnarray*}
a &=&(xa^{2})y=(x(aa))y=(a(xa))y=(y(xa))a \\
&=&(y(x((xa^{2})y)))a=(y((xa^{2})(xy)))a=((xa^{2})(y(xy)))a \\
&=&(((xy)y)(a^{2}x))a=((y^{2}x)(a^{2}x))a=(a^{2}((y^{2}x)x))a \\
&=&(a^{2}(x^{2}y^{2}))a=(a(x^{2}y^{2}))a^{2}=(a(x^{2}y^{2}))(aa)\in
(AS)A^{2}.
\end{eqnarray*}

Thus $(AS)A^{2}=A.$ Now by using medial law with left identity$,$ left
invertive law$,$ paramedial law and medial law$,$ we have%
\begin{eqnarray*}
a &=&(xa^{2})y=(x(aa))y=(a(xa))y=(y(xa))a \\
&=&(y(xa))((xa^{2})y)=(xa^{2})((y(xa))y)=(x(aa))((y(xa))y) \\
&=&(a(xa))((y(xa))y)=(((y(xa))y)(xa))a=((ax)(y(y(xa))))a \\
&=&((((xa^{2})y)x)(y(y(xa))))a=(((xy)(xa^{2}))(y(y(xa))))a \\
&=&(((xy)y)((xa^{2})(y(xa))))a=((y^{2}x)((x(aa))(y(xa))))a \\
&=&((y^{2}x)((xy)((aa)(xa))))a=((y^{2}x)((aa)((xy)(xa))))a \\
&=&((aa)((y^{2}x)((xy)(xa))))a=((aa)((y^{2}x)((xx)(ya))))a \\
&=&((((xx)(ya))(y^{2}x))(aa))a=((((ay)(xx))(y^{2}x))(aa))a \\
&=&((((x^{2}y)a)(y^{2}x))(aa))a=(((xy^{2})(a(x^{2}y)))(aa))a \\
&=&((a((xy^{2})(x^{2}y)))(aa))a=((a(x^{3}y^{3}))(aa))a \\
&\in &((AS)A^{2})A\subseteq AA=A^{2}.
\end{eqnarray*}

Hence $A^{2}=A.$

$(ii)$ $\Longrightarrow (i)$ is obvious.
\end{proof}

\begin{lemma}
\cite{nav}\label{ag}Every non empty subset $A$ of an intra-regular
LA-semigroup $S$ with left identity is a left ideal of $S$ if and only if it
is a right ideal of $S$.
\end{lemma}

\begin{theorem}
\label{12}In an intra-regular LA-semigroup $S$ with left identity, the
following conditions are equivalent.
\end{theorem}

$(i)$ $A$ is a $(1,2)$-ideal of $S$.

$(ii)$ $A$ is a two-sided ideal of $S.$

\begin{proof}
$(i)$ $\Longrightarrow (ii):$ Assume that $S$ is intra-regular LA-semigroup
with left identity and let $A$ be a $(1,2)$-ideal of $S$ then, $%
(AS)A^{2}\subseteq A.$ Let $a\in A$, then since $S$ is intra-regular so
there exist $x,$ $y\in S$ such that $a=(xa^{2})y.$ Now by using medial law
with left identity$,$ left invertive law and paramedial law$,$ we have%
\begin{eqnarray*}
sa &=&s((xa^{2})y)=(xa^{2})(sy)=(x(aa))(sy)=(a(xa))(sy) \\
&=&((sy)(xa))a=((sy)(xa))((xa^{2})y)=(xa^{2})(((sy)(xa))y) \\
&=&(y((sy)(xa)))(a^{2}x)=a^{2}((y((sy)(xa)))x) \\
&=&(aa)((y((sy)(xa)))x)=(x(y((sy)(xa))))(aa) \\
&=&(x(y((ax)(ys))))(aa)=(x((ax)(y(ys))))(aa) \\
&=&((ax)(x(y(ys))))(aa)=((((xa^{2})y)x)(x(y(ys))))(aa) \\
&=&(((xy)(xa^{2}))(x(y(ys))))(aa)=(((a^{2}x)(yx))(x(y(ys))))(aa) \\
&=&((((yx)x)a^{2})(x(y(ys))))(aa)=(((y(ys))x)(a^{2}((yx)x)))(aa) \\
&=&(((y(ys))x)(a^{2}(x^{2}y)))(aa)=(a^{2}(((y(ys))x)(x^{2}y)))(aa) \\
&=&((aa)(((y(ys))x)(x^{2}y)))(aa)=(((x^{2}y)((y(ys))x))(aa))(aa) \\
&=&(a((x^{2}y)(((y(ys))x)a)))(aa)\in (AS)A^{2}\subseteq A.
\end{eqnarray*}

Hence $A$ is a left ideal of $S$ and by Lemma \ref{ag}, $A$ is a two-sided
ideal of $S.$

$(ii)$ $\Longrightarrow (i):$ Let $A$ be a two-sided ideal of $S$. Let $y\in
(AS)A^{2},$ then $y=(as)b^{2}$ for some $a,b\in A$ and $s\in S.$ Now by
using medial law with left identity$,$ we have%
\begin{equation*}
y=(as)b^{2}=(as)(bb)=b((as)b)\in AS\subseteq A.
\end{equation*}

Hence $(AS)A^{2}\subseteq A$, therefore $A$ is a $(1,2)$-ideal of $S$.
\end{proof}

\begin{lemma}
\cite{nav} \label{iffff}Let $S$ be an LA-semigroup, then $S$ is
intra-regular if and only if every left ideal of $S$ is idempotent.
\end{lemma}

\begin{lemma}
\cite{nav}\label{qqq}Every non empty subset $A$ of an intra-regular
LA-semigroup $S$ with left identity is a two-sided ideal of $S$ if and only
if it is a quasi ideal of $S$.
\end{lemma}

\begin{theorem}
A two-sided ideal of an intra-regular LA-semigroup $S$ with left identity is
minimal if and only if it is the intersection of two minimal two-sided
ideals of $S$.

\begin{proof}
Let $S$ be intra-regular LA-semigroup and $Q$ be a minimal two-sided ideal
of $S$, let $a\in Q$. As $S(Sa)\subseteq Sa$ and $S(aS)\subseteq a(SS)=aS,$
which shows that $Sa$ and $aS$ are left ideals of $S,$ so by Lemma \ref{ag}, 
$Sa$ and $aS$ are two-sided ideals of $S$.

Now%
\begin{eqnarray*}
S(Sa\cap aS)\cap (Sa\cap aS)S &=&S(Sa)\cap S(aS)\cap (Sa)S\cap (aS)S \\
&\subseteq &(Sa\cap aS)\cap (Sa)S\cap Sa\subseteq Sa\cap aS.
\end{eqnarray*}

This implies that $Sa\cap aS$ is a quasi ideal of $S,$ so by using \ref{qqq}%
, $Sa\cap aS$ is a two-sided ideal of $S$.

Also since $a\in Q$, we have%
\begin{equation*}
Sa\cap aS\subseteq SQ\cap QS\subseteq Q\cap Q\subseteq Q\text{.}
\end{equation*}

Now since $Q$ is minimal, so $Sa\cap aS=Q,$ where $Sa$ and $aS$ are minimal
two-sided ideals of $S$, because let $I$ be an two-sided ideal of $S$ such
that $I\subseteq Sa,$ then $I\cap aS\subseteq Sa\cap aS\subseteq Q,$ which
implies that $I\cap aS=Q.$ Thus $Q\subseteq I.$ Therefore, we have%
\begin{equation*}
Sa\subseteq SQ\subseteq SI\subseteq I,\text{ gives }Sa=I.
\end{equation*}

Thus $Sa$ is a minimal two-sided ideal of $S$. Similarly $aS$ is a minimal
two-sided ideal of $S.$

Conversely, let $Q=I\cap J$ be a two-sided ideal of $S$, where $I$ and $J$
are minimal two-sided ideals of $S,$ then by using \ref{qqq}, $Q$ is a quasi
ideal of $S$, that is $SQ\cap QS\subseteq Q.$ Let $Q^{^{\prime }}$ be a
two-sided ideal of $S$ such that $Q^{^{\prime }}\subseteq Q$, then%
\begin{equation*}
SQ^{^{\prime }}\cap Q^{^{\prime }}S\subseteq SQ\cap QS\subseteq Q,\text{
also }SQ^{^{\prime }}\subseteq SI\subseteq I\text{ and }Q^{^{\prime
}}S\subseteq JS\subseteq J\text{.}
\end{equation*}

Now 
\begin{equation*}
S(SQ^{^{\prime }})=\left( SS\right) (SQ^{^{\prime }})=(Q^{^{\prime
}}S)\left( SS\right) =(Q^{^{\prime }}S)S=\left( SS\right) Q^{^{\prime
}}=SQ^{^{\prime }},
\end{equation*}%
which implies that $SQ^{^{\prime }}$ is a left ideal and hence a two-sided
ideal by Lemma \ref{ag}. Similarly $Q^{^{\prime }}S$ is a two-sided ideal of 
$S$.

But since $I$ and $J$ are minimal two-sided ideals of $S$, therefore $%
SQ^{^{\prime }}=I$ and $Q^{^{\prime }}S=J.$ But $Q=I\cap J,$ which implies
that, $Q=SQ^{^{\prime }}\cap Q^{^{\prime }}S\subseteq Q^{^{\prime }}.$ This
give us $Q=Q^{^{\prime }}$ and hence $Q$ is minimal.
\end{proof}
\end{theorem}


\begin{thebibliography}{99}
\bibitem{clif} A. H. Clifford and G. B. Preston,\emph{\ }\textit{The
algebraic theory of semigroups}, John Wiley \& Sons, $(vol.1)$ $1961$.

\bibitem{p.hol} P. Holgate,\textit{\ }Groupoids satisfying a simple
invertive law, The Math. Stud., $1-4,61$ $(1992),$ $101-106$.

\bibitem{kaz} M. A. Kazim and M. Naseeruddin, On almost semigroups, The
Alig. Bull. Math., $2$ $(1972),1-7.$

\bibitem{Madad} Madad Khan, Some studies in AG$^{\ast }$-groupoids, Ph. D.,
thesis, Quaid-i-Azam University, Islamabad, Pakistan, $2008$.

\bibitem{nav} Madad Khan and N. Ahmad, Characterizations of left almost
semigroups by their ideals, Journal of Advanced Research in Pure
Mathematics, $2$ $(2010),$ $61-73.$

\bibitem{Mus3} Q. Mushtaq and S. M. Yousuf, On LA-semigroups, The\ Alig.
Bull. Math., $8$ $(1978),$ $65-70.$

\bibitem{Mushtaq and Yusuf} Q. Mushtaq and S. M. Yusuf, On LA-semigroup
defined by a commutative inverse semigroups, Math. Bech., $40$ $(1988)$, $%
59-62$.

\bibitem{madadkhan2} Q. Mushtaq and M. Khan, Ideals in left almost
semigroups, Proceedings of 4th International Pure Mathematics Conference, $%
2003,$ $65-77.$

\bibitem{Naseeruddin} M. Naseeruddin, Some studies in almost semigroups and
flocks, Ph.D., thesis, Aligarh Muslim University, Aligarh, India, $1970.$

\bibitem{ref10} P. V. Proti\'{c} and N. Stevanovi\'{c}, AG-test and some
general properties of Abel-Grassmann's groupoids, PU. M. A., $4$, $6$ $%
(1995) $, $371-383$.

\bibitem{Protic} N. Stevanovi\'{c} and P. V. Proti\'{c}, Composition of
Abel-Grassmann's 3-bands, Novi Sad, J. Math., $2$, $34$ $(2004)$, $175-182$.
\end{thebibliography}
\end{document}